\documentclass[12pt]{article}
\usepackage{graphicx}
\usepackage{color}
\usepackage{amsmath}
\usepackage{amsthm}
\newtheorem{statement}{statement}[section]

\newtheorem{corollary}[statement]{Corollary}

\newtheorem{lemma}[statement]{Lemma}

\newtheorem{theorem}[statement]{Theorem}
\catcode`\Ž=\active \def Ž{\'e}
\catcode`\=\active \def {\`e}
\catcode`\ˆ=\active \def ˆ{\`a}
\catcode`\=\active \def {\`u}
\catcode`\"=\active \def "{\`\i}
\catcode`\˜=\active \def ˜{\`o}
\catcode`\'=\active \def '{\"e}
\catcode`\Š=\active \def Š{\"a}
\catcode`\•=\active \def •{\"\i}
\catcode`\Ÿ=\active \def Ÿ{\"u}
\catcode`\š=\active \def š{\"o}
\catcode`\=\active \def {\^e}
\catcode`\‰=\active \def ‰{\^a}
\catcode`\"=\active \def "{\^\i}
\catcode`\™=\active \def ™{\^o}
\catcode`\ž=\active \def ž{\^u}
\catcode`\=\active \def {\c c}
\catcode`\'=\active \def '{\c C}
\catcode`\Ï=\active \def Ï{\oe}
\font\sevenrm=cmr10 scaled 700

\font\ninerm=cmr10 scaled 900
\font\tenrm=cmr10 scaled 1000
\font\fourteenbf=cmbx10 scaled 1400

\font\sectionbf=cmbx10 scaled 1900

\def\smallskip{\vspace{5pt}}
\def\medskip{\vspace{10pt}}

\definecolor{magenta}{rgb}{1,0.8,1}
\definecolor{yellow}{rgb}{1,1,0}
\definecolor{green}{rgb}{0.4,1,0.6}
\definecolor{red}{rgb}{1,0.4,0.4}
\definecolor{blue}{rgb}{0.6,0.8,1} 

\def\m{\colorbox{magenta}}
\def\y{\colorbox{yellow}}
\def\g{\colorbox{green}}

\def\Ext{\hbox{\rm Ext\ \!}}
\def\Int{\hbox{\rm Int\ \!}}
\def\tbs{\overline{\theta}\raisebox{0.5pt}{$^*$}\!\!\!} 
\def\tbss{\overline{\theta}\raisebox{0.5pt}{$\scriptstyle ^*$}\!\!\!} 
\def\Tbs{\overline{\Theta}\raisebox{0.5pt}{$^*$}\!\!\!} 
\def\Tbss{\overline{\Theta}\raisebox{0.5pt}{$\scriptstyle ^*$}\!\!\!} 

\begin{document}

\title{\vspace{-3cm}
\textbf{The Tutte Polynomial\\ of a Morphism of Matroids.\\}
{\fourteenbf 6. A Multi-Faceted Counting Formula\\
\vspace{-8pt} for Hyperplane Regions and Acyclic Orientations}}

\author{Michel Las Vergnas$^{\scriptscriptstyle\diamondsuit}$}
\date{\ninerm May 21, 2012$^{\star}$}

\maketitle

\vspace{2cm}

\begin{abstract}
\tenrm
\baselineskip 12pt
We show that the 4-variable generating function of certain orientation related parameters of an ordered oriented matroid is the evaluation at $(x+u,y+v)$ of its Tutte polynomial.
This evaluation contains as special cases the counting of regions in hyperplane arrangements and of acyclic orientations in graphs.
Several new 2-variable expansions of the Tutte polynomial of an oriented matroid follow as corollaries.

This result hold more generally for oriented matroid perspectives,
with specific special cases the counting of bounded regions in hyperplane arrangements or of bipolar acyclic orientations in graphs.
\par

In corollary, we obtain expressions for the partial derivatives of the Tutte polynomial as generating functions of the same orientation parameters.

\end{abstract}

\vfill
{\tenrm
\baselineskip 12pt
\vskip 0.1cm
\noindent

\par}

\bigskip
\hrule height .5pt depth 0pt width 6truecm
\medskip

{\ninerm
\noindent $\scriptscriptstyle ^\diamondsuit$ C.N.R.S., Paris\par
\noindent $^\star$ submitted to J. Combinatorial Theory ser. B, Feb. 18, 2012}

\eject

\noindent
{\sectionbf Introduction}

\medskip\noindent
Let $M$ be a matroid on a set $E$.
The Tutte polynomial $t(M;x,y)$ of $M$, equivalent to the generating function for cardinality and rank of subsets of $E$, can be defined by the
closed formula

\medskip\noindent
\hspace{1cm}
$\displaystyle t(M;x,y)=\sum_{A\subseteq E}(x-1)^{r(M)-r_M(A)}(y-1)^{|A|-r_M(A)}$\hfill (1)

\smallskip\noindent
where $r_M(A)$ denotes the rank of $A$ in $M$.

\medskip
As well-known, the Tutte polynomial of a matroid on a linearly ordered set can also be expressed as the generating function of {\it Tutte activities} - internal and external - of bases, 
providing a state model with numerous applications \cite{BO92}\cite{Cr69}\cite{EMM10}\cite{Tu54}.

\medskip\noindent
\hspace{1cm}
$\displaystyle t(M;x,y)=\sum_{B\subseteq E{\hbox{\sevenrm\ basis of }}M}x^{\iota_M(B)}y^{\epsilon_M(B)}$ \hfill (2)

\bigskip
We have introduced in \cite{LV84} a state model for the Tutte polynomial of an oriented matroid on a linearly ordered set as a generating function of {\it orientation activities}.

\medskip\noindent
\hspace{1cm}
$\displaystyle t(M;x,y)=
\sum_{A\subseteq E}
\bigl({x\over 2}\bigr)^{{o^*}_{M}(A)}
\bigl({y\over 2}\bigr)^{o^{\phantom{*}}_M(A)}$ \hfill (3)

\smallskip
Basic definitions and properties of oriented matroids can be found in  \cite{OM99}.

We point out that formula (3) contains as special cases, quoted here in increasing order of generality, the counting $t(2,0)$ of acyclic orientations of graphs by R. Stanley \cite{St73}, of acyclic orientations of regular matroids by T. Brylawski and D. Lucas \cite{BL76}, of regions in hyperplane arrangements by R.O. Winder \cite{Wi66} and T. Zaslavsky \cite{Za75}, of acyclic reorientations of oriented matroids by the author \cite{LV75}.

\medskip
A comparison of the state models (2) and (3) for the Tutte polynomial - one in terms of Tutte activities of bases, the other in terms of orientation activities of subsets - has been a motivation for a series of papers by E. Gioan and the author on the so-called {\it active bijection} \cite{GiLV05}\cite{GiLV07}\cite{GiLV09}\cite{GiLV12a}\cite{GiLV12b}.
The active bijection relates the two types of activities by means of activity preserving mappings with prescribed multiplicities.

Surprisingly enough, the active bijection establishes a relationship between the Tutte polynomial and linear programming \cite{GiLV12b}.

\bigskip
In the paper \cite{GT90}, G. Gordon and L. Traldi exhibit an expansion of the Tutte polynomial of an ordered matroid in terms of Tutte activities  similar to (4)(see Example 3.3).

\medskip\noindent
\hspace{1cm}
$\displaystyle t(M;x,y)=
\sum_{A\subseteq E}
\bigl({x\over 2}\bigr)^{\iota_{M}(A)}
\bigl({y\over 2}\bigr)^{\epsilon_M(A)}$ \hfill (4)

\medskip
In \cite{GT90}, the formula (4) is derived  from a 4-variable expansion for the Tutte polynomial in terms of (generalized) Tutte activities.
One may wonder whether an analogous 4-variable expansion also exists in terms of orientation activities.

\medskip
It turns out that such an expansion does exist.
Its existence and properties constitute the object of the present note.

\medskip
Actually, the formula in \cite{LV84} is given for objects more general than Tutte polynomials of oriented matroids, namely for certain 2-variable specializations of the 3-variable Tutte polynomials of oriented matroid perspectives.
The same level of generality holds here.

The 3-variable {\it Tutte polynomial of a  matroid perspective} has been introduced by the author in 1975 \cite{LV75}.
We have studied its properties in a series of papers: fundamental properties in \cite{LV80} \cite{LV99}, Eulerian partitions in surfaces \cite{LV81}, activities of orientations in \cite{LV84}\cite{LV12}, vectorial matroids in \cite{EtLV04}, computational complexity in \cite{LV07}.

The present 4-variable expansion - see below Theorem \ref{thm:4p} - refines the 2-variable expansion of \cite{LV84} Theorem 3.1.

Let $M\rightarrow M'$ be an oriented matroid perspective on a linearly ordered set $E$. We have

\medskip\noindent
\hspace{1cm}
$\displaystyle t(M,M';x,y,1)=
\sum_{A\subseteq E}
\bigl({x\over 2}\bigr)^{{o^*}_{M'}(A)}
\bigl({y\over 2}\bigr)^{o^{\phantom{*}}_M(A)}$ \hfill (5)

\medskip
Among applications of (5) specific to perspectives are self-dual forms of the counting of bounded regions in hyperplane arrangements, or of bipolar acyclic orientations in graphs, generalizing to oriented matroids results of \cite{GZ83} - see \cite{LV77}\cite{LV84}.

\medskip
In \cite{LV12}, the expansion of $t(x+u,y+v)$ in terms of Tutte activities is obtained from expressions of the partial derivatives as generating functions.
Here, we go the reverse way. 
In Section 3, we obtain expressions of partial derivatives from the 4-variable expansion.

We show on an example how this computation is related to the {\it active partitions} discussed in \cite{GiLV05}\cite{GiLV07}\cite{GiLV12a}.
Details will be published later.

\bigskip

\bigskip\bigskip\noindent
{\sectionbf Matroid Perspectives}

\medskip
Matroid perspectives generalize linear mappings of vector spaces.
For the convenience of the reader, we recall here the relevant definitions and main properties.

\medskip
Let $M$, $M'$ be two matroids on a same set $E$.
We say that they constitute a {\it matroid perspective}, 
denoted by $M\rightarrow M'$, when the identity map on $E$ is a matroid strong map,
that is, if at least one, hence all, of the following equivalent properties (i)-(iv) holds

\smallskip
(i) Any circuit of $M$ is a union of circuits of $M'$.

\smallskip
(ii) Any cocircuit of $M'$ is a union of cocircuits of $M$.

\smallskip
(iii) No circuit of $M$ and cocircuit of $M'$ meet in exactly one element.

\smallskip
(iv) For all $Y\subseteq X\subseteq E$ we have
$$r_M(Y)-r_{M'}(Y)\leq r_M(X)-r_{M'}(X)$$

\medskip
As well-known, a matroid perspective factorizes: 
there is a matroid $N$ on $F$, $E\subseteq F$, 
such that $M=N\setminus (E\setminus F)$ and $M'=N/(E\setminus F)$.

\bigskip
When $M$ $M'$ are two oriented matroids, we say that they constitute an {\it oriented matroid perspective} if at least one, hence all, of the following equi\-valent properties (i$'$)-(iii$'$) holds.

\smallskip
(i$'$) Any circuit of $M$ is a conformal union of (signed) circuits of $M'$.

\smallskip
(ii$'$) Any cocircuit of $M'$ is a conformal union of  (signed) cocircuits of $M$.

\smallskip
(iii$'$) No (signed) circuit of $M$ and (signed) cocircuit of $M'$ have a non empty conformal intersection.

\medskip
Two signed sets $Y\subseteq X$ are {\it conformal} if $Y^+\subseteq X^+$ and $Y^-\subseteq X^-$. 

\medskip
From a topological point of view, property (ii$'$) expresses that the vertices of a pseudohyperplane arrangement representing $M'$ belong to faces of a pseudohyperplane arrangement representing $M$.

Oriented matroid perspectives do not factorize in general as shown by J. Richter-Gebert in \cite{RG93} Corollary 3.5.

\bigskip
The Tutte polynomial of a matroid perspective $M\rightarrow M'$ is defined in \cite{LV99} by the closed formula

\smallskip\noindent
$\displaystyle t(M,M';x,y,z)=$\hfill\\
\null\hfill $\displaystyle =\sum_{A\subseteq E}(x-1)^{r(M')-r_{M'}(A)}(y-1)^{|A|-r-M(A)}z^{r(M)-r(M')-(r_M(A)-r_{M'}(A))}$

\smallskip
Property (iii) ensures that we have indeed non negative powers of $z$.

\section{O-activities and $\Theta$-activities}

Let $M$ be an oriented matroid of a linearly ordered set $E$.

\medskip
The following definitions have been introduced in \cite{LV84}.
An element of $E$ is {\it orientation active} resp. {\it orientation dually-active} in $M$ if it is smallest in some positive circuit resp. positive cocircuit of $M$.
We denote by $O(M)$ the set of orientation active elements of $M$, and by $O^*(M)$ its set of orientation dually-active elements.
We have $O^*(M)=O(M^*)$.

\medskip
We set
$$o(M)=|O(M)|$$
$$o^*(M)=|O^*(M)|$$

\medskip
For short, we say that $o(M)$ and $o^*(M)$ are the {\it o-activities} of $M$.

\medskip
For $A\subseteq E$, we denote by $-_AM$ the oriented matroid obtained from $M$ by reorientation on $A$.
Note that $-_{E\setminus A}M=-_AM$.

\medskip
We refine the o-activities as follows. Set 
$$\Theta_M(A)=O(-_AM)\setminus A$$ 
$$\overline{\Theta}_M(A)=\Theta_M(E\setminus A)=O(-_AM)\cap A$$ 
$$\theta_M(A)=|\Theta_M(A)|$$
$$\overline{\theta}_M(A)=|\overline{\Theta}_M(A)|$$

We have 
$$o_M(A)=\theta_M(A)+\overline{\theta}_M(A)$$

Dually, we set 
$$\Theta^*_M(A)=\Theta_{M^*}(A)=O^*(-_AM)\setminus A$$
$$\Tbs_M(A)=\overline{\Theta}_{M^*}(A)=O^*(-_AM)\cap A$$
$$\theta^*_M(A)=|\Theta^*_M(A)|$$
$$\tbs_M(A)=|\Tbs_M(A)|$$

We have
$$o^*_M(A)=\theta^*_M(A)+\tbs_M(A)$$

\medskip
The 4 parameters 
$\theta_M(A)$, $\overline{\theta}_M(A)$,
$\theta^*_M(A)$, $\tbs_M(A)$,
depending on $A$,
are the 4 $\theta$-{\it activities} of $M$.

\section{A 4-Expansion for the Tutte Polynomial}

We have shown in \cite{LV84} Theorem 3.1 that the generating function of the two $o$-activities of the reorientations of an oriented matroid is the evaluation of its Tutte polynomial at $(2x,2y)$.
Our main result here is that the generating function of the four $\theta$-activities is also an evaluation of the Tutte polynomial.

\bigskip
\begin{theorem}
\label{thm:4m}

Let $M$ be an oriented matroid on a linearly ordered set $E$.
We have 
$$t(M;x+u,y+v)=\sum_{A\subseteq E}
x^{{\theta^*}\!\!\!_M(A)}
u^{\tbss_M(A)}
y^{\theta^{\phantom{*}}_M(A)}
v^{\overline{\theta}^{\phantom{*}}_M(A)}$$
where $t(M;x,y)$ denotes the Tutte polynomial of $M$.
\end{theorem}



\begin{figure}[h] 
{\bf Example 1}\par
\centerline{\includegraphics[scale=0.35]{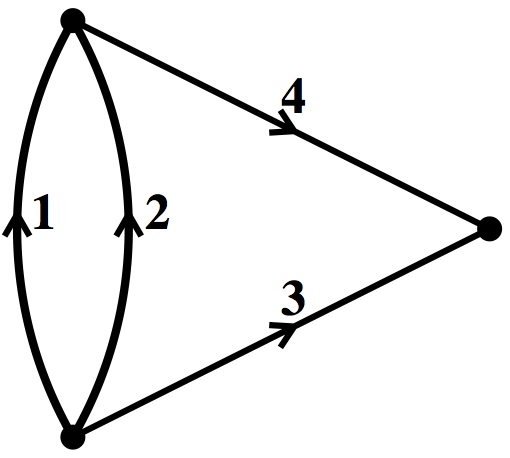}}
\end{figure}

\scalebox{0.8}{
The Tutte polynomial of the cycle matroid of the above 4-edge directed graph 
is}\par
\centerline{\scalebox{0.8}{ $t(x,y)=x^2+xy+y^2+x+y$.}}

\medskip
\centerline{
\(
\begin{array}{|l||l|l||l|l|l|l||c|}
 \hline
\scalebox{0.75}{$A$} & 
\scalebox{0.75}{$O^*(A)$} & 
\scalebox{0.75}{$O(A)$} & 
\scalebox{0.75}{$\Theta^*(A)$} & 
\scalebox{0.75}{$\Tbs\ (A)$} & 
\scalebox{0.75}{$\Theta(A)$} & 
\scalebox{0.75}{$\overline{\Theta}(A)$} & 
\scalebox{0.75}{$
x^{|\Theta^*(A)|}u^{|\Tbss\ (A)|}
y^{|\Theta(A)|}v^{|\overline{\Theta}(A)|}$}
\\ \hline\hline
\emptyset 
     & 13 &    & 13 &    &    &    & \m{$x^2$} \\ \hline
4    & 1  &    & 1  &    &    &    & x \\ \hline
3    &    & 12 &    &    & 12 &    & \g{$y^2$} \\ \hline
34   & 13 &    & 1  & 3  &    &    & \m{$xu$} \\ \hline
2    & 3  & 1  & 3  &    & 1  &    & \y{$xy$} \\ \hline
23   &    & 1  &    &    & 1  &    & y \\ \hline
24   &    & 12 &    &    & 1  & 2  & \g{$yv$}  \\ \hline
234  & 3  & 1  &    & 3  & 1  &    & \y{$uy$}  \\ \hline
1    & 3  & 1  & 3  &    &    & 1  & \y{$xv$}  \\ \hline
14   &    & 1  &    &    &    & 1  & v   \\ \hline
13   &    & 12 &    &    & 2  & 1  & \g{$yv$}  \\ \hline
134  & 3  & 1  &    & 3  &    & 1  & \y{$uv$}  \\ \hline
12   & 13 &    & 3  & 1  &    &    & \m{$xu$}  \\ \hline
124  &    & 12 &    &    &    & 12 & \g{$v^2$} \\ \hline
123  & 1  &    &    & 1  &    &    & u   \\ \hline
1234 & 13 &    &    & 13 &    &    & \m{$u^2$} \\ \hline
\end{array}
\)
}

\medskip
\scalebox{0.8}{
\indent
In accordance with Theorem \ref{thm:4m}, the last column sums up to}

\vskip -1.5pt
\centerline{\scalebox{0.8}{
$t(x+u,y+v)=$\m{$(x+u)^2$}+\y{$(x+u)(y+v)$}+\g{$(y+v)^2$}$+(x+u)+(y+v)$.}}

\vskip -1.5pt
\scalebox{0.8}{
\medskip\indent
n.b. Replacing $A$ by $E\setminus A$ exchanges $x$ and $u$ on one hand, and $y$ and $v$ on the other.}

\vskip -3pt
\scalebox{0.8}{
Hence it would suffice to compute the lines $A$ such that $1\notin A$.}

\smallskip
\centerline{Table 1}

\medskip
Theorem \ref{thm:4m} will be proved in Section 3, as a special case of the more general Theorem \ref{thm:4p}.

\medskip
Theorem \ref{thm:4m} is the orientation counterpart of the result of G. Gordon and L. Traldi for Tutte activities \cite{GT90}, quoted in the introduction as (4) (see also \cite{LV12} Theorem 2.9).
The relationship between orientation activities and Tutte activities is the object of a series of papers by E. Gioan and the author 
\cite{GiLV05}\cite{GiLV07}\cite{GiLV09}\cite{GiLV12a}: the active mapping from reorientations to bases preserves pair of activities.
Theorem \ref{thm:4m} allows to precise the properties of the active mapping:
the associated active bijection from reorientations to subsets preserves 4-uplets of activities.

\bigskip
Theorem \ref{thm:4m} is closely related to the {\it active partition} associated with a basis of an ordered oriented matroid.
The notion of active partition has been introduced by E. Gioan and the author in terms of graphs in \cite{GiLV05}.
The case of general oriented matroids will appear in detail in \cite{GiLV12a} (an extended abstract can be found in \cite{GiLV07}).

With a basis $B$ of $M$ is associated the term $t(M;B;x,y)=x^{\iota_M(B)}y^{\epsilon_M(B)}$ of the basis expansion of its Tutte polynomial.
The active partition associated with $B$ is a partition of $E$ into $\iota_M(B)+\epsilon_M(B)$ classes.
Each class of the active partition is associated with one $B$-active element, either internally or internally.
By reorienting any orientation associated with $B$ on arbitrary  unions of classes of the active partition, we get all the $2^{\iota_M(B)+\epsilon_M(B)}$ reorientations associated with $B$ by the active mapping.
The reorientations associated in this way with $B$ are exactly those yielding the terms of the expansion $t(M;B;x+u,y+v)$ in Theorem \ref{thm:4m}
associated with the subsets in the Dawson interval defined by $B$.

A proof of Theorem \ref{thm:4m} in terms of active partition will appear in \cite{GiLV12a}. 

\bigskip
We recover immediately the matroid case of \cite{LV84} by specializing $x=u$ $y=v$ in Theorem \ref{thm:4m}.

\medskip
\begin{corollary}[\cite{LV84} Theorem 3.1]
Let $M$ be an ordered oriented matroid on a set $E$.
We have
$$t(M;2x,2y)=\sum_{A\subseteq E}x^{o^*(-_AM)}y^{o(-_AM)}$$
\end{corollary}

\bigskip
As well-known, the evaluation $t(2,0)$ of the Tutte polynomial counts the number acyclic orientations of graph resp. of regions of a hyperplane arrangement, of acyclic reorientations of an oriented matroid \cite{LV75}\cite{St73}\cite{Wi66}\cite{Za75}.
Setting $x=1$ $u=1$ in the formula of Theorem \ref{thm:4m}, we get indeed that $t(2,0)$ is the number of subsets $A\subseteq E$ such that $-_AM$ has orientation activity 0, i.e. is acyclic.
Setting $x=2$ $u=0$ resp. $x=0$ $u=2$, we get two alternate expansions of $t(2,0)$, in terms of orientation $\theta$-activities.

\medskip
\begin{corollary}
\label{cor:t20-expansion}
Let $M$ be an oriented matroid on a linearly ordered set $E$.
We have
$$t(M;2,0)
=\sum_{\buildrel {A\subseteq E} 
\over {\scriptstyle \tbss_M(A)
=\theta^{\phantom{*}}_M(A)
=\overline{\theta}^{\phantom{*}}_M(A)=0}}
2^{\!\!\ \theta^*\!\!\!_M(A)}
=\sum_{\buildrel {A\subseteq E} 
\over {\scriptstyle \theta^*\!\!\!_M(A)
=\theta^{\phantom{*}}_M(A)
=\overline{\theta}^{\phantom{*}}_M(A)=0}}
2^{\!\!\ \tbss_M(A)}
$$

\end{corollary}

\bigskip
\section{Generalization to Perspectives}

Theorem \ref{thm:4m} is a special case of a more general theorem dealing with Tutte polynomials of oriented matroid perspectives \cite{LV75}\cite{LV99}.

\medskip
\begin{theorem}
\label{thm:4p}
Let $M\rightarrow M'$ be an oriented matroid perspective on a linearly ordered set $E$.
We have 
$$t(M,M';x+u,y+v,1)=\sum_{A\subseteq E}
x^{{\theta^*}\!\!\!_{M'}(A)}
u^{\tbss_{M'}(A)}
y^{\theta^{\phantom{*}}_M(A)}
v^{\overline{\theta}^{\phantom{*}}_M(A)}$$
where $t(M,M;x,y,z)$ denotes the 3-variable Tutte polynomial of $M\rightarrow M'$.
\end{theorem}


\begin{figure}[h] 
{\bf Example 2}\par
\centerline{\includegraphics[scale=0.35]{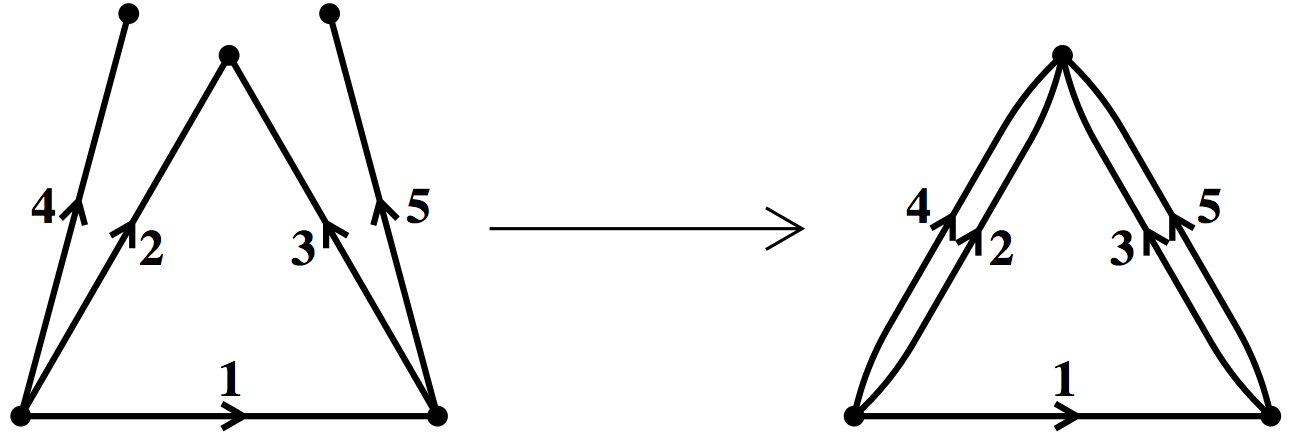}}
\end{figure}

\smallskip
\scalebox{0.8}{
Let us consider the oriented matroid perspective given by the cycle matroids of the above}\par
\noindent
\raisebox{2pt}{
\scalebox{0.8}{graphs. The evaluation at $z=1$ of its Tutte polynomial is}}\par
\centerline{\scalebox{0.8}{$t(x,y,1)=x^2+5x+4y+10$.}}\par
\scalebox{0.8}{The $\theta$-activities are shown on the following table, together with the contributed terms.}
\medskip

\centerline{
\scalebox{.75}{
\(
\begin{array}{||l||l|l||l|l|l|l||c|||l||l|l||l|l|l|l||c|||}
 \hline
\scalebox{0.75}{$A$} & 
\scalebox{0.75}{$O^*(A)$} & 
\scalebox{0.75}{$O(A)$} & 
\scalebox{0.75}{$\Theta^*(A)$} & 
\scalebox{0.75}{$\Tbs\ (A)$} & 
\scalebox{0.75}{$\Theta(A)$} & 
\scalebox{0.75}{$\overline{\Theta}(A)$} & 
\scalebox{0.75}{$t(A)$} &
\scalebox{0.75}{$A$} & 
\scalebox{0.75}{$O^*(A)$} & 
\scalebox{0.75}{$O(A)$} & 
\scalebox{0.75}{$\Theta^*(A)$} & 
\scalebox{0.75}{$\Tbs\ (A)$} & 
\scalebox{0.75}{$\Theta(A)$} & 
\scalebox{0.75}{$\overline{\Theta}(A)$} & 
\scalebox{0.75}{$t(A)$} 
\\ \hline\hline
\emptyset 
     & 12 &    & 12 &    &    &    & \m{$x^2$}  & 1     & 12 &    & 2  & 1  &    &    & \m{$xu$} \\ \hline
5    & 1  &    & 1  &    &    &    & \y{$x$}    & 15    &    &    &    &    &    &    & 1  \\ \hline
4    &    &    &    &    &    &    & 1    & 14    & 1  &    &    & 1  &    &    & \y{$u$}  \\ \hline
45   &    &    &    &    &    &    & 1    & 145   &    &    &    &    &    &    & 1  \\ \hline
3    & 1  &    & 1  &    &    &    & \y{$x$}    & 13    &    & 1  &    &    &    & 1  & \g{$v$}  \\ \hline
35   & 1  &    & 1  &    &    &    & \y{$x$}    & 135   &    & 1  &    &    &    & 1  & \g{$v$}  \\ \hline
34   &    &    &    &    &    &    & 1    & 134   &    & 1  &    &    &    & 1  & \g{$v$}  \\ \hline
345  & 1  &    & 1  &    &    &    & \y{$x$}    & 1345  &    & 1  &    &    &    & 1  & \g{$v$}  \\ \hline
2    &    & 1  &    &    & 1  &    & \g{$y$}    & 12    & 1  &    &    & 1  &    &    & \y{$u$}  \\ \hline
25   &    & 1  &    &    & 1  &    & \g{$y$}    & 125   &    &    &    &    &    &    & 1  \\ \hline
24   &    & 1  &    &    & 1  &    & \g{$y$}    & 124   & 1  &    &    & 1  &    &    & \y{$u$}  \\ \hline
245  &    & 1  &    &    & 1  &    & \g{$y$}    & 1245  & 1  &    &    & 1  &    &    & \y{$u$}  \\ \hline
23   &    &    &    &    &    &    & 1    & 123   &    &    &    &    &    &    & 1  \\ \hline
235  & 1  &    & 1  &    &    &    & \y{$x$}    & 1235  &    &    &    &    &    &    & 1  \\ \hline
234  &    &    &    &    &    &    & 1    & 1234  & 1  &    &    & 1  &    &    & \y{$u$}  \\ \hline
2345 & 12 &    & 1  & 2  &    &    & \m{$xu$}   & 12345 & 12 &    &    & 12 &    &    & \m{$u^2$} \\
\hline 
\end{array}
\)
}}

\medskip
\scalebox{0.8}{
\indent
In accordance with Theorem \ref{thm:4p}, the columns $t(A)$ sum up to}

\vskip -1.5pt
\scalebox{0.8}{
$t(x+u,y+v,1)=$\m{$(x+u)^2$}+\y{$5(x+u)$}+\g{$4(y+v)$}$+10$.}

\smallskip
\centerline{Table 2}

\medskip
As above in Section 2, we recover the main result of \cite{LV84} in the case of matroid perspectives by the specialization $x=u$ $y=v$ 

\medskip
\begin{corollary}[\cite{LV84} Theorem 3.1]
Let $M\rightarrow M'$ be an ordered oriented matroid perpective on a set $E$.
We have
$$t(M,M';2x,2y,1)=\sum_{A\subseteq E}x^{o^*(-_AM')}y^{o(-_AM)}$$
\end{corollary}

\bigskip
We recall that given an oriented matroid perspective on a set $E$, the evaluation $t(M,M';0,0,1)$ counts the number of subsets $A\subseteq E$ such that $-_AM$ is acyclic and $-_AM'$ is totally cyclic \cite{LV77}.
This statement can be considered as self dual form of the counting of acyclic reorientations by $t(M;2,0)$.
In the particular case $M'=M/e\oplus{\bf 0}(e)$, where $e$ is a non factor element  of $M$ - neither loop or isthmus, and ${\bf 0}(e)$ the rank 0 matroid on $e$, then $t(M,M';0,0,1)$ counts the {\it number of bounded regions of a hyperplane arrangement} with infinity at $e$, or of bipolar acyclic orientations defined by the edge $e$ in a graph \cite{LV77}\cite{LV84}.

\bigskip
Evaluating Theorem \ref{thm:4p} at $x=y=-1$ and $u=v=1$, we get alternate expansions of $t(M,M';0,0,1)$.

\begin{corollary}
$$t(M,M';0,0,1)=\sum_{A\subseteq E}
(-1)^{{\theta^*}\!\!\!_{M'}(A)+\theta^{\phantom{*}}_M(A)}$$
\end{corollary}

Three similar formulas are obtained for the other different choices of -1 and 1 for the variables.

\bigskip
As of today, the generalization to oriented matroid perspectives of the active bijection and of active partitions by E. Gioan and the author is still in progress.
Hence, the possible proof of Theorem \ref{thm:4m} by methods along these lines mentioned in Section 2 cannot be used for Theorem \ref{thm:4p}.
It turns out that the proof by deletion/contraction of \cite{LV84} requires only slight adjustments to establish Theorem \ref{thm:4p}.

\medskip
For $a\in E$, set $o^*(M;a)=1$ if $a$ is internally active in $M$  and $o^*(M;a)=0$ otherwise, $o(M;a)=1$ if $a$ is externally active in $M$  and $o(M;a)=0$ otherwise.

The main step of the proof of Theorem \ref{thm:4p} is the following lemma
from \cite{LV84} (see Lemma 3.2).

\medskip
\begin{lemma}
\label{lem:l1}
Let $M\rightarrow M'$ be an oriented matroid on a linearly ordered set $E$, with greatest element $e$.
Then (i) or (ii) holds, where

\medskip\noindent
(i) $o^*(M';a)=o^*(M'\setminus e;a)$, $o(M;a)=o(M\setminus e;a)$,
$o^*(-_eM';a)=o^*(M'/e;a)$ and $o(-_eM;a)=o(M/e;a)$ for all $a\in E\setminus\{e\}$,

\medskip\noindent
(ii) $o^*(M';a)=o^*(M'/e;a)$, $o(M;a)=o(M/e;a)$, $o^*(-_eM';a)=o^*(M'\setminus e;a)$
and $o(-_eM;a)=o(M\setminus e;a)$ for all $a\in E\setminus\{e\}$.
\end{lemma}

\medskip
For the convenience of the reader, we reproduce the proof of Lemma \ref{lem:l1}.

\medskip
\begin{proof}
The proof is broken into several steps.

\smallskip
Let $a\in E\setminus \{e\}$.

\medskip
(1) {\it $o(M;a)=1$ or $-_eo(M;a)=1$ implies $o(M/e;a)=1$}.

\smallskip
Let $X$ be a positive circuit of $M$ resp. $-_eM$ with least element $a$.
By a property of contraction in oriented matroids, there is a positive circuit $X'$ of $M/e=(-_eM)/e$ such that $a\in X'\subseteq X$, 
hence $o(M/e;a)=1$.

\medskip
(2) {\it $o(M;a)=1$ and $-_eo(M;a)=1$ implies $o(M\setminus e;a)=1$.}

\smallskip
There exist signed circuits $X$ and $Y$ of $M$ both with least element $a$ such that $X^-=\emptyset$
and $Y^-\subseteq\{e\}$.
Suppose $o(M\setminus e;a)=0$: Then necessarily $e\in X^+\cap Y^-$, hence by the elimination property in oriented matroids, there exists
a positive circuit $Z$ such that $a\in Z\subseteq(X\cup Y)\setminus\{e\}$.
We get $o(M\setminus e;a)=1$, a contradiction.

\medskip
(3) {\it $o(M\setminus e;a)=1$ implies $o(M;a)=-_eo(M;a)=1$.}

\smallskip
The proof is immediate.

\medskip
(4) {\it $o(M/e;a)=1$ implies $o(M;a)=1$ or $-_eo(M;a)=1$.}

\smallskip
Let $X'$ be a positive circuit of $M/e$ with least element $a$.
There exists a signed circuit $X$ of $M$ such that $X'=X\setminus\{e\}$.
Since $e$ is the greatest element of $E$, the element $a$ also smallest in $X$.
We have $X^-\subseteq\{e\}$, hence $X$ is a positive circuit of $M$ or $-_eX$ is a positive circuit of $-_eM$. 

\medskip
(5) {\it $o(M;a)=o(M\setminus e;a)=1$ if and only if $o(-_eM;a)=o(M/e;a)$.}

\smallskip
Suppose $o(M;a)=o(M\setminus e;a)=1$, but $o(-_eM;a)\not=o(M/e;a)$.
We have $o(-_eM;a)=0$ since $o(-_eM;a)=1$ implies $o(M/e;a)=1$  by (1), and $o(M/e;a)=1$.
Then $o(M;a)=1$ by (4), hence $o(M\setminus e;x)=a$, which contradicts (3) since $o(-_eM;a)=0$.

Conversely, suppose $o(-_eM;a)=o(M/e;a)$, but $o(M;a)\not=o(M\setminus e;a)$.
We have $o(M\setminus e;a)=0$ since $o(M\setminus e;a)=1$ implies $o(M;a)=1$ by (3), and $o(M;a)=1$.
Then $o(M/e;a)=1$ by (1), hence $o(-_eM;a)=1$.
But this, together with $o(M;a)=1$ contradicts (2) since $o(M\setminus e;a)=0$.

\medskip
(6) {\it If, for some $b\in E\setminus\{e\}$, we have $o(M;b)\not=o(-_eM;b)$ and $o(M;b)=o(M\setminus e;b)$,
then $o(M;a)=o(M\setminus e;a)$ for all $a\in E\setminus\{e\}$.}

\smallskip
Suppose $o(M;a)\not=o(M\setminus e;a)$ for some $a\in E\setminus\{e,b\}$.
Then $o(M\setminus e;x)=0$ and $o(M;x)=1$ (since $o(M\setminus e;x)=1$ implies $o(M;a)=1$ by (3)).
Hence there is a positive circuit $X$ of $M$ with least element $a$, and necessarily $e\in X$.
On the other hand, $o(M\setminus e;b)=0$ since $o(M\setminus e;b)=1$ implies $o(M;b)=o(-_eM;b)$ by (3).
Hence $o(M;b)=0$, $o(-_eM;b)=1$, and there is a signed circuit $Y$ of $M$ with least element $b$ such that
$Y^-\subseteq\{e\}$.
Necessarily, $e\in X^+\cap Y^-$, since $o(M\setminus e;a)=o(M\setminus e;b)=0$.
Set $c=\hbox{Min}(a,b)$.
By the elimination property in oriented matroids, there is a positive circuit $Z$ of $M$ such that $c\in Z\subseteq(X\cup Y)\setminus\{e\}$.
Clearly, $c$ is the least element of $Z$.
Hence $o(M\setminus e;c)=1$, a contradiction.

\medskip
(7) $o(M;a)=o(-_eM;a)${\it implies} $o(M;a)=o(-_eM;a)=o(M\setminus e;a)=o(M/e;a)$.

\smallskip
Suppose $o(M;a)=o(-_eM;a)$.
If $o(M;a)=0$, then $o(M\setminus e;a)=0$ by (3).
If $o(M;a)=1$, then $o(M\setminus e;a)=1$ by (2).
In both cases, we have $o(M;a)=o(M\setminus e;a)$.
It follows that $o(-_eM;a)=o(M/e;a)$ by (5).

\medskip
(8) {\it We have $o(M;a)=o(M\setminus e;x)$ and $o(-_eM;x)=o(M/e;a)$ for all $a\in E\setminus\{e\}$,
or $o(M;a)=o(M/e;a)$ and $o(-_eM;a)=o(M\setminus e;a)$ for all $a\in E\setminus\{e\}$.}

\smallskip
If $o(M;a)=o(-_eM;a)$ for all $a\in E\setminus\{e\}$, 
then $o(M;a)=o(-_eM;a)=o(M\setminus e;a)=o(M/e;a)$ for all $a\in E\setminus\{e\}$ by (7).
Suppose there is $b\in E\setminus\{e\}$ such that $o(M;b)\not=o(-_eM;b)$.
We have $o(M;y)=o(M\setminus e;y)$ or $o(-_eM;b)=o(M\setminus e;b)$.
Suppose for instance $o(M;b)=o(M\setminus e;b)$.
We have $o(M;a)=o(M\setminus e;a)$ for all $a\in E\setminus\{e\}$ by (6), hence $o(M;a)=o(M/e;a)$ for all $a\in E\setminus\{e\}$ by (5).

\medskip
(9) $o(M;a)=1$ {\it implies} $o^*(M';a)=0$.

\smallskip
Suppose $o(M;a)=1$:
there is a positive circuit $X$ of $M$ with least element $X$.
Since $M\rightarrow M'$ is an oriented matroid perspective, there is a positive circuit $X'$ of $M'$
such that $a\in X'\subseteq X$, hence there is no positive cocircuit of $M'$ containing $a$ by the orthogonality property in oriented matroids, implying $o^*(M';a)=0$.

\medskip
(10) {\it If $o(M;a)\not=o(-_eM;a)$, then $o(M;a)=o(M\setminus e;a)$ implies $o^*(M';a)=o^*(M'/e;a)$.}

\smallskip
Suppose $o^*(M';a)\not=o^*(M'\setminus e;a)$ under the hypothesis of (10).
By ($1^*$) (i.e. by (1) applied to $M^*$) $o^*(M';a)=1$ implies $o^*(M'\setminus e;a)=1$.
Hence $o^*(M';a)=0$ and $o^*(M'\setminus e;a)=1$.
Now by ($4^*$) $o^*(M';a)=1$ or $o^*(-_eM';a)=1$.
Hence $o^*(-_eM';a)=1$.
Then, by (9) $o(-_eM;a)=0$.
It follows that $o(M;a)=1$, hence $o(M\setminus e;a)=1$, by hypothesis, contradicting $o(-_eM;a)=0)$ by (3).

\medskip
(11) We now prove Lemma \ref{lem:l1}.

\smallskip
If $o^*(M';a)=o^*(-_eM';a)$ for all $a\in E\setminus\{e\}$, or  $o(M;a)=o(-_eM;a)$ for  $a\in E\setminus\{e\}$,
then (11) follows clearly from (7), ($7^*$) and (8), ($8^*$).
If there is $a\in E\setminus\{e\}$ such that  $o^*(M';a)\not=o^*(-_eM';a)$ and  $o(M;a)\not=o(-_eM;a)$, 
then Lemma \ref{lem:l1} follows from (8), ($8^*$) and (10).

The remaining possibility is that for all $a\in E\setminus\{e\}$ we have either $o^*(M';a)\not=o^*(-_eM';a)$,
$o(M;a)=o(-_eM;a)$, or $o^*(M';a)=o^*(-_eM';a)$, $o(M;a)\not=o(-_eM;a)$, and both cases occur.
Replacing if necessary, $M\rightarrow M'$ by  $-_eM\rightarrow -_eM'$, we may suppose notation such that
$o(M;b)=o(M\setminus e;b)$ for some $b\in E\setminus\{e\}$ with $o(M;b)\not=o(-_eM;b)$.
Then by (6) $o(M;a)=o(M\setminus e;a)$ for all $a\in E\setminus\{e\}$.

If $a\in E\setminus\{e\}$ is such that $o(M;a)\not=o(-_eM;a)$ we have $o^*(M';a)=o^*(M'\setminus e;a)$ by (10).

Consider now $a\in E\setminus\{e\}$ such that $o(M;a)=o(-_eM;a)$.
We have $o^*(M';a)\not=o^*(-_eM';a)$ by our hypothesis.
If $o^*(M';a)\!\!=\!\!0$, we have \linebreak
$o^*(-_eM';a)=1$, hence there is a cocircuit $X'$ of $M'$ with $X^-=\{e\}$.
By our hypothesis there is $b\in E\setminus\{e\}$ with $o(M;b)=0$, $o(-_eM;b)=1$, implying
the existence of a circuit $Y$ of $M$ with $Y^-=\{e\}$.
Since $M\rightarrow M'$ there is a circuit $Y'$ of $M'$ with $Y'^-=\{e\}$.
Then $X'$ and $Y'$ contradict the orthogonality property.
Therefore $o^*(M';a)=1$ and $o^*(-_eM';a)=0$.
Now by ($1^*$) $o^*(M'\setminus e;a)=1$, hence $o^*(M';a)=o^*(M'\setminus e;a)$.

Thus, for all $a\in E\setminus\{e\}$ we have $o^*(M';a)=o^*(M'\setminus e;a)$ and $o(M;a)=o(M\setminus e;a)$.
Therefore, Lemma \ref{lem:l1} follows from (5) and ($5^*$), or from (8) and ($8^*$).
\end{proof}

\medskip
Set
$$f(M,M';A;x,u,y,v)=
x^{\theta^*_{M'}(A)}
u^{\tbss_{M'}(A)}
y^{\theta^{\phantom{*}}_M(A)}
v^{\overline{\theta}^{\phantom{*}}_M(A)}$$
and
$$f(M,M';x,u,y,v)=\sum_{A\subseteq E}f(M,M';A:x,u,y,v)$$

\medskip
\begin{proof}[Proof of Theorem \ref{thm:4p}]
We distinguish several cases

\smallskip\noindent
(i) {\it $e$ is neither an isthmus of $M'$ nor a loop  of  $M$}\par

\smallskip
Summing up the equalities of Lemma \ref{lem:l1} for all $a\in A$ resp. $a\in 
E\setminus\{e\}\setminus A$, 
and observing that  $o^*(M';e)=o^*(-_eM';e)=o(M;e)=o(-_eM;e)=0$, since $e$ is the greatest element of $E$, 
we get that

\medskip
\(
\begin{array}{ll}
\theta^*_{M'}(A)=\theta^*_{M'\setminus e}(A) &
\tbs_{M'}(A)=\tbs_{M'\setminus e}(A)\\
\theta_M(A)=\theta_{M\setminus e}(A) &
\overline{\theta}_M(A)=\overline{\theta}_{M\setminus e}(A)\\
\theta^*_{-_e{M'}}(A)=\theta^*_{M'/e}(A) &
\tbs_{-_e{M'}}(A)=\tbs_{M'/e}(A)\\
\theta_{-_eM}(A)=\theta_{M/e}(A) &
\overline{\theta}_{-_eM}(A)=\overline{\theta}_{M/e}(A)\\
\end{array}
\)

\medskip
or

\medskip
\(
\begin{array}{ll}
\theta^*_{M'}(A)=\theta^*_{M'/e}(A) &
\tbs_{M'}(A)=\tbs_{M'/ e}(A)\\
\theta_M(A)=\theta_{M/e}(A) &
\overline{\theta}_M(A)=\overline{\theta}_{M/e}(A)\\
\theta^*_{-_e{M'}}(A)=\theta^*_{M'\setminus e}(A) &
\tbs_{-_e{M'}}(A)=\tbs_{M'\setminus e}(A)\\
\theta_{-_eM}(A)=\theta_{M\setminus e}(A) &
\overline{\theta}_{-_eM}(A)=\overline{\theta}_{M\setminus e}(A)\\
\end{array}
\)

\medskip
It follows that

\smallskip\noindent
$\displaystyle f(M,M';A;x,u,y,v)+f(-_eM,-_eM';A;x,u,y,v)=$\hfill\\
\null\hfill $\displaystyle =f(M\setminus e,M'\setminus e;A;x,u,y,v)+f(M/e,M'/e;A;x,u,y,v)$

\smallskip
Summing up for $A\subseteq E\setminus\{e\}$ we get
$$f(M,M';x,u,y,v)=
f(M\setminus e,M'\setminus e;x,u,y,v)+f(M/e,M'/;x,u,y,v)$$

\medskip\noindent
(ii) {\it $e$ is an isthmus of $M'$ (hence also an isthmus of $M$)}

\smallskip
For $A\subseteq E\setminus\{e\}$, we have readily

\smallskip
\(
\begin{array}{ll}
\theta^*_{M'}(A)=\theta^*_{M'\setminus e}(A)+1 &
\tbs_{M'}(A)=\tbs_{M'\setminus e}(A)\\
\theta_M(A)=\theta_{M\setminus e}(A)&
\overline{\theta}_M(A)=\overline{\theta}_{M\setminus e}(A)\\
\end{array}
\)

\smallskip
and 

\smallskip
\(
\begin{array}{ll}
\theta^*_{M'}(A\cup\{e\})=\theta^*_{M'\setminus e}(A) &
\tbs_{M'}(A\cup\{e\})=\tbs_{M'\setminus e}(A)+1\\
\theta_M(A\cup\{e\})=\theta_{M\setminus e}(A)&
\overline{\theta}_M(A\cup\{e\})=\overline{\theta}_{M\setminus e}(A)\\
\end{array}
\)

\smallskip
It follows that
$$f(M,M';A;x,u,y,v)=xf(M\setminus e,M'\setminus e;A;x,u,y,v)$$
and
$$f(M,M';A\cup\{e\};x,u,y,v)=uf(M\setminus e,M'\setminus e;A;x,u,y,v)$$

Therefore,

\smallskip\noindent
$\displaystyle f(M,M';x,u,y,v)=\sum_{A\subset E}f(M,M';A;x,u,y,v)=$\hfill\\
\smallskip
\null\hspace{0.2cm}
\scalebox{0.9}{
$\displaystyle =\sum_{A\subseteq E\setminus\{e\}}f(M,M';A;x,u,y,v)+
\sum_{e\in A\subseteq E}f(M,M';A;x,u,y,v)=$}\hfill\\
\smallskip
\null\hspace{0.2cm}
\scalebox{0.9}{
$\displaystyle =\sum_{A\subseteq E\setminus\{e\}}xf(M\setminus e,M'\setminus e;A;x,u,y,v)+
\sum_{A\subseteq E\setminus\{e\}}uf(M\setminus e,M'\setminus e;A;x,u,y,v)=$}\hfill\\
\smallskip
\null\hspace{0.2cm}
\scalebox{0.9}{
$\displaystyle =xf(M\setminus e,M'\setminus e;x,u,y,v)+
uf(M\setminus e,M'\setminus e;x,u,y,v)=$}\hfill\\
\smallskip
\null\hspace{0.3cm}
$\displaystyle =(x+u)f(M\setminus e,M'\setminus e;x,u,y,v)$

\medskip\noindent
(iii) {\it $e$ is a loop of $M$ (hence also a loop of $M'$)}

\smallskip
As in (ii), we have, dually
$$f(M,M';x,u,y,v)=(y+v)f(M\setminus e,M'\setminus e;x,u,y,v)$$

\medskip\noindent
(iv) We have readily
$$f(\emptyset,\emptyset;x,u,y,v)=1$$

\medskip
Properties (i)-(iv) show that $f(M,M';x,u,y,v)$ verifies the deletion/\break contraction inductive relations satisfied by
$t(M,M';x+u,y+v,1)$.
Therefore, by \cite{LV99} Theorem 5.3, we have
$$f(M,M';x,u,y,v)=t(M,M';x+u,y+v,1)$$

\end{proof}

\bigskip
As first shown by G. Gordon and L. Traldi (see \cite{GT90} Examples 3.1-3.5), and extended by the author to matroid perspectives (see \cite{LV12} Proposition 2.9), many 2-variable expansions of the Tutte polynomial follow readily from an expansion as a 4-variable generating function similar to Theorem \ref{thm:4p}, but in terms of Tutte activities.

The most remarkable are obtained by setting some of $x$ $u$ $y$ $v$ to either 0 or 1, and/or replacing by $x/2$ $y/2$, and performing an appropriate change of variables. 
A total of 25 expansions, 9 different up to reordering, could be thus obtained from Theorem \ref{thm:4p}.

Here, we limit ourselves to three of them, referring the reader to \cite{LV12} for a complete list in the case of Tutte activities.

\bigskip
\begin{corollary}
{\rm M. Las Vergnas 1984 \cite{LV84}}
Let $M$ be an oriented matroid on a linearly ordered set $E$.
We have 
$$t(M;x,y)=
\sum_{A\subseteq E}
\bigl({x\over 2}\bigr)^{{o^*}_{M'}(A)}
\bigl({y\over 2}\bigr)^{o^{\phantom{*}}_M(A)}$$
\end{corollary}

\begin{proof}
Replace $x$ and $u$ by $x\over 2$, and $y$ and $v$ by $y\over 2$
in Theorem \ref{thm:4m}.
\end{proof}

\begin{corollary}
Let $M$ be an oriented matroid on a linearly ordered set $E$.
We have 
$$t(M;x,y)=
\sum_{A\subseteq E}
(x-1)^{\theta^*_M(A)}
(y-1)^{\theta^{\phantom{*}}_M(A)}$$
\end{corollary}

\begin{proof}
Replace $x$ by $x-1$, $u$ by 1, $y$ by $y-1$, and $v$ by $1$
in Theorem \ref{thm:4m}.
\end{proof}

\begin{corollary}
Let $M$ be an oriented matroid on a linearly ordered set $E$.
We have 
$$t(M;x,y)=
\sum_{\buildrel {A\subseteq E}\over
{\buildrel\scriptstyle {\tbss_M(A)=0}
\over{\scriptstyle {\overline{\theta}^{\phantom{*}}_M(A)=0}}}}
x^{\theta^*_M(A)}
y^{\theta^{\phantom{*}}_M(A)}$$
\end{corollary}

\begin{proof}
Replace $u$ and $v$ by 0 in Theorem \ref{thm:4m}.
\end{proof}

\begin{corollary}
Let $M$ be an oriented matroid perspective on a linearly ordered set $E$.
The number of subsets $A\subseteq E$ such that 
$\theta^*_M(A)=\theta^{\phantom{*}}_M(A)=0$ 
is equal to the number of bases of $M$.
\end{corollary}

The same property holds for the three other classes of subsets obtained by duality and exchange of $A$ and $E\setminus A$.

\medskip
The above corollaries generalize straightforwardly to oriented matroid perspectives.

\section{Derivatives}

In \cite{LV12}, we have used state models of Tutte polynomial partial derivatives in terms of internal and external activities to obtain an expansion of $t(M,M';x+u,y+v,z)$.
Here, we go the reverse way, using the expansion of $t(M,M';x+u,y+v,1)$ in terms of orientations given by Theorem \ref{thm:4p} to obtain expansions of partial derivatives.

\begin{theorem}
\label{thm:perspective_derivative}
Let $M\rightarrow M'$ be an oriented matroid perspective on a linearly ordered set $E$. 
We have
$${{\partial^{p+q} t}\over{{\partial x^p}{\partial y^q}}}                                         
(M,M';x,y,1)
=p!q!\sum_{\buildrel {A\subseteq E}\over
{\buildrel\scriptstyle {\tbss_{M'}(A)=p}\over{\scriptstyle {\overline{\theta}_{M}(A)=q}}}}
x^{\theta^*_{M'}(A)}y^{\theta_M(A)}$$
\end{theorem}

\begin{proof}
By Taylor formula, we have
$$t(M,M';x+u,y+v,1)=\sum_{p\geq 0\ q\geq 0}{1\over{p!q!}}u^pv^q
{{\partial^{p+q} t}\over{{\partial x^p}{\partial y^q}}}(M,M';x,y,1)$$

Theorem \ref{thm:perspective_derivative} follows readily from the expansion of $t(M,M';x+u,y+v,1)$ given by Theorem \ref{thm:4p}.
\end{proof}
 
We can easily obtain three alternative expansions.
It suffices to apply differently Taylor formula, for instance with respect to $x$ and $v$ instead of $u$ and $v$.

\medskip
\begin{corollary}
\label{cor:perspective_(x,x)_derivative}
\noindent
Let $M\rightarrow M'$ be a matroid perspective on a linearly ordered set $E$, and $p$ be a non negative integer.
Then
$${{d^p t}\over{dx^p}}(M,M';x,x,1)=
p!\sum_
{\buildrel {A\subseteq E}\over
{\scriptstyle \tbss_{M'}(A)+\overline{\theta}_M(A)=p}}
x^{\theta^*_{M'}(A)+\theta_M(A)}$$
\end{corollary}

\begin{proof}
By Taylor formula, and Theorem \ref{thm:perspective_derivative} we have
$$t(M,M';x+u,y+v,1)
=\sum_{p\geq 0 q\geq 0}u^pv^q\sum_{\buildrel {A\subseteq E}\over
{\buildrel\scriptstyle {\tbs_{M'}(A)=p}\over{\scriptstyle {\overline{\theta}_{M}(A)=q}}}}
x^{\theta^*_{M'}(A)}y^{\theta_M(A)}$$
Hence
$$t(M,M';x+u,x+u,1)
=\sum_{p\geq 0 q\geq 0}u^{p+q}\sum_{\buildrel {A\subseteq E}\over
{\scriptstyle \tbss_{M'}(A)+\overline{\theta}_M(A)=p}}
x^{\theta^*_{M'}(A)+\theta_M(A)}$$
By Taylor formula again
$${{d^p t}\over{dx^p}}(M,M';x+u,x+u,z)=\sum_{k\geq 0}{1\over k!}u^k{{d^p t}\over{dx^p}}(M,M';x,x,z)$$
Corollary \ref{cor:perspective_(x,x)_derivative} follows.
\end{proof}

\bigskip
It turns out that the subsets yielding the expansion of the partial derivatives, that is the subsets $A$ such that $\tbs_{M'}(A)=p$ and 
$\overline{\theta}_{M}(A)=q$, can be simply described. 
In \cite{LV12}, an analogous description for internal and external activities was provided by the Dawson partitions associated to bases or independent/spanning sets.
Here, the key tool are the active partitions, already mentioned in Section 2.

\medskip
As of today, active partitions are fully available only for oriented matroids, the generalization to oriented matroid perspectives being still a work in progress.
Details of the construction for oriented matroids will be given in \cite{GiLV12a}.
We present it briefly on the example of Section 2.

\medskip\noindent
{\bf Example 1} (continued)

\medskip
\centerline{
\scalebox{0.8}{
\(
\begin{array}{|c||c|c|c|c|c|}
\hline
\hbox{basic orientations}&
\begin{array}{c} \phantom{} \\ \phantom{} \end{array} \emptyset \begin{array}{c} \phantom{} \\ \phantom{} \end{array} &2&3&4&23 \\ \hline
\hbox{active edges}& 13\ *&3\ *\ 1&*\ 12& 1\ * & *\ 1 
\begin{array}{c}\phantom{} \\ \phantom{} \end{array} \\ \hline
\hbox{active partition}& 12+34\ *&34\ *\ 12&*\ 1+234& 1234\ * & *\ 1234 
\begin{array}{c}\phantom{} \\ \phantom{} \end{array} \\ \hline
\hbox{active orientation partition}& \emptyset\ 12\ 34\ 1234 &2\ 234\ 1\ 134 & 3\ 13\ 24\ 124 & 4\ 123 & 14\ 23 
\begin{array}{c}\phantom{} \\ \phantom{} \end{array} \\ \hline\hline
t(x,y)=x^2+xy+y^2+x+y&
\begin{array}{c} \emptyset \\ x^2 \end{array}&
\begin{array}{c} 2 \\ xy \end{array}&
\begin{array}{c} 3 \\ y^2 \end{array}&
\begin{array}{c} 4 \\ x \end{array}&
\begin{array}{c} 23 \\ y \end{array}\\ \hline
{{\partial t}\over{\partial x}}(x,y)=2x+y+1&
\begin{array}{cc} 
  \begin{array}{c} 12 \\ xu \end{array}&
  \begin{array}{c} 34 \\ xu \end{array} 
\end{array}&
\begin{array}{c}  1 \\ yu \end{array} &
\begin{array}{c} \\ \end{array} &
\begin{array}{c} 123 \\ u \end{array} &
\begin{array}{c} \\ \end{array}\\ \hline
{{\partial t}\over{\partial y}}(x,y)=x+2y+1&&
\begin{array}{c} 234 \\ xv \end{array}&
\begin{array}{cc} 
  \begin{array}{c} 13 \\ yv \end{array}&
  \begin{array}{c} 24 \\ yv \end{array} 
\end{array}&
\begin{array}{c} \\ \end{array}&
\begin{array}{c} 14 \\ v \end{array}\\ \hline
{1\over 2}{{\partial^2 t}\over{\partial x^2}}(x,y)=1&
\begin{array}{c} 1234 \\ u^2 \end{array}&
\begin{array}{c} \\ \end{array}&
\begin{array}{c} \\ \end{array}&
\begin{array}{c} \\ \end{array}&
\begin{array}{c} \\ \end{array}\\ \hline
{{\partial^2 t}\over{\partial x\partial y}}(x,y)=1&
\begin{array}{c} \\ \end{array}&
\begin{array}{c} 134 \\ uv \end{array}&
\begin{array}{c} \\ \end{array}&
\begin{array}{c} \\ \end{array}&
\begin{array}{c} \\ \end{array}\\ \hline
{1\over 2}{{\partial^2 t}\over{\partial x^2}}(x,y)=1&
\begin{array}{c} \\ \end{array}&
\begin{array}{c} \\ \end{array}&
\begin{array}{c} 124 \\ v^2 \end{array}&
\begin{array}{c} \\ \end{array}&
\begin{array}{c} \\ \end{array}\\ \hline
\end{array}
\)
}}

\medskip
\scalebox{0.8}{
\indent
Table 3 here has to be compared with Table 1 of \cite{LV12}.}

\vskip -3pt
\scalebox{0.8}{
\indent
Similarity is obvious.
The precise relationship made explicit by the active bijection,}

\vskip -3pt
\scalebox{0.8}{
 a 1-1 mapping $2^E\mapsto 2^E$, which provides theorems and algorithms relating the 4}
 
\vskip -3pt
\scalebox{0.8}{
$\theta$-activities to the 4 Tutte activities.}

\centerline{Table 3}

\medskip
We read on Table 1 the 5 orientations $A$ with 
$\tbs\ (A)=\overline{\theta}(A)=0$, namely $\emptyset$ 4 3 2 23.
These orientations have the role played by bases in the Dawson partitions considered in \cite{LV12}.
We call them {\it basic orientations}.

\smallskip
The pairs of dual and primal orientation activities are respectively (2,0) (1,0) (0,2) (1,1) (0,1).
We compute the active partitions as in \cite{GiLV05}, obtaining $12+34*$, 
$1234*$, $*1+234$, $12*34$, $*1234$.
The symbol $*$ separates the dually-active classes (on its left) from the primally-active classes (on its right).
We obtain the orientations in each case by reversing unions of classes of the active partition, i.e. taking symmetric differences, in all possible ways.
The {\it active orientation partition} is the partition of the set of $2^{|E|}$ orientations obtained in this way.

\medskip
Knowing the activities of the basic orientations and which classes have been reversed yields the values of the 4 $\theta$-activities.
When we reverse a dually-active resp. primally-active class, $\theta^*$ decreases by 1 and $\tbs\ $ increases by 1 resp. $\theta$ decreases by 1 and $\overline{\theta}$ increases by 1.
Therefore, the monomial $x^iu^jy^kv^\ell$ associated with the basic orientation is multiplied by $x^{-1}u$ as many times as the number of reversed dually-active classes and by $y^{-1}v$ as many times as the number of reversed primally-active classes.\par

\medskip
More precisely, let $B$ be a basis of $M$, and $X\subseteq\Int_M(B)$, $Y\subseteq\Ext_M(B)$. 
Let $A_B$ be any reorientation in the inverse image of $B$ by the active mapping - produced by an algorithm \cite{GiLV05}\cite{GiLV07}\cite{GiLV09}\cite{GiLV12a}.
Let $A$ be the reorientation obtained from $A_B$ by reversing all elements in the union of the classes of the active partition activated by $X\cup Y$.
Then $A'=B\setminus X\cup Y$ is the subset in the Dawson interval defined by $B$ associated with $A$ by the active bijection.
Using notation of \cite{LV12}, $\theta$-activities of $A$ and Tutte activities of $A'$ are related as follows: 
$\theta^*_M(A)\!=\!cr_M(A')$\ $\tbs_M(A)\!=\!\iota_M(A')$\  
$\theta_M(A)\!=\!nl_M(A')$\ $\overline{\theta}_M(A)\!=\!\epsilon_M(A')$.


\bigskip
\noindent
Michel Las Vergnas\\
Universit\'e Pierre et Marie Curie (Paris 6)\\
case 247 - Institut de Math\'ematiques de Jussieu\\
Combinatoire \& Optimisation\\ 
4 place Jussieu, 75252 Paris cedex 05 (France)\\

\noindent
{\it mlv@math.jussieu.fr}


\begin{thebibliography}{10}

\bibitem{OM99}
A. Bj\"orner, M. Las Vergnas, B. Sturmfels, N. White and  G. Ziegler,
\newblock {\it Oriented matroids}, 2nd edition. 
\newblock Encyclopedia of Mathematics and its Applications 46, 
Cambridge University Press, Cambridge, UK 1999.

\bibitem{BL76}
T. Brylawski, D. Lucas, 
\newblock Uniquely representable combinatorial geometries.
\newblock Colloq. Int. Teorie Combinatorie (Roma 1973), B. SegrŽ ed., Atti dei Convegni Lincei 17 Tomo 1 (Roma 1976), 83--104.

\bibitem{BO92}
T. Brylawski, J. Oxley, 
\newblock The Tutte Polynomial and its Applications.
\newblock Chapter 6 in: N. White (ed.), {\it Matroid Applications}, Cambridge University Press 1992, 123--225.

\bibitem{Cr69}
H.H. Crapo,
\newblock The Tutte polynomial.
\newblock Aequationes Math. 3 (1969), 211--229.

\bibitem{GiLV05}
E. Gioan and  M. Las Vergnas, 
\newblock Activity preserving bijections between spanning trees and orientations in graphs,
\newblock (Special issue  FPSAC 2002)
Discrete Math. 298 (2005), 169-188.

\bibitem{EtLV04}
G. Etienne, M. Las Vergnas,
\newblock The Tutte polynomial of a morphism of matroids.
4. Vectorial matroids.
\newblock Advances in Applied Mathematics 32 (2004), 198--211.

\bibitem{GiLV07}
E. Gioan, M. Las Vergnas, 
\newblock Fully optimal bases and the active bijection in graphs, hyperplane arrangements, and oriented matroids,
\newblock  (Proceedings EuroComb, Sevilla 2007) Electronic Notes in Discrete Mathematics {29} (2007), 365-371.

\bibitem{GiLV09}
E. Gioan, M. Las Vergnas,
\newblock The active bijection in graphs, hyperplane arrangements, and oriented matroids, 1. The fully optimal basis of a bounded region.
\newblock European J. Combinatorics, 30 (2009), 1868--1886.

\bibitem{GiLV12a}
E. Gioan,  M. Las Vergnas, 
\newblock The active bijection in graphs, hyperplane arrangements, and oriented matroids, 2. Decomposition of activities, 
\newblock in preparation.

\bibitem{GiLV12b}
E. Gioan, M. Las Vergnas, 
\newblock The Active bijection in graphs, hyperplane arrangements, and oriented matroids, 3. Linear Programming,
\newblock in preparation.

\bibitem{EMM10}
J.A. Ellis-Monaghan, C. Merino, 
\newblock Graph Polynomials and Their Applications I: The Tutte Polynomial.
\newblock Matthias Dehmer ed., {\it Structural Analysis of Complex Networks}, BirkhaŸser 2010, 219--256.

\bibitem{GT90}
G. Gordon, L. Traldi, 
\newblock Generalized activities and the Tutte polynomial.
\newblock Discrete Maths. 85 (1990), 167--176

\bibitem{GZ83}
C. Greene, T. Zaslavsky,
\newblock On the interpretation of Whitney numbers through arrangements of hyperplanes, zonotopes, non-Radon partitions, and orientations of graphs. 
\newblock Trans. Amer. Math. Soc. 280 (1983), 97Ð-126.

\bibitem{LV75} 
M. Las Vergnas, 
\newblock Matro{\"\i}des orientables, 
\newblock C. R. Acad. Sci. Paris S\'er. A 280 (1975), 61-65.

\bibitem{LV75}
M. Las Vergnas,
\newblock Extensions normales d'un matro{\"\i}de, polyn\^ome de Tutte d'un morphisme.
\newblock C.R. Acad. Sci. Paris s\'er. A 280 (1975), 1479--1482.

\bibitem{LV77}
M. Las Vergnas, 
\newblock Acyclic and totally cyclic orientations of combinatorial geometries. 
\newblock Discrete Math. 20 (1977/78), 51Ð-61.

\bibitem{LV80}
M. Las Vergnas
\newblock On the Tutte polynomial of a morphism of matroids.
\newblock Annals Discrete Mathematics 8 (1980), 7--20.

\bibitem{LV81}
M. Las Vergnas,
\newblock Eulerian circuits of 4-valent graphs imbedded in surfaces.
\newblock in: {\it Algebraic methods in Graph Theory}, Proc Coll. Math. Soc. J{\'a}nos Bolyai 25 (Szeged, Hungary, 1978), North-Holland 1981, 451--477

\bibitem{LV99}
M. Las Vergnas,
\newblock The Tutte polynomial of a morphism of matroids.
1. Set pointed matroids and matroid perspectives.
\newblock Annales de l'Institut Fourier 40 (1999), 973--1015.

\bibitem{LV84} 
M. Las Vergnas, 
\newblock The Tutte polynomial of a morphism of matroids.
2. Activities of orientations.
\newblock {\it Progress in Graph Theory} (Proc. Waterloo Silver Jubilee Conf. 1982), J.A. Bondy \& U.S.R. Murty eds.,
Academic Press, Toronto 1984, 367--380.

\bibitem{LV07}
M. Las Vergnas,
\newblock The Tutte polynomial of a morphism of matroids.
4. Computational complexity.
\newblock Portugaliae Mathematica 64 (2007), 303--309.

\bibitem{LV12} 
M. Las Vergnas, 
\newblock The Tutte polynomial of a morphism of matroids 5. 
Derivatives as generating functions of Tutte activities.
\newblock European J. Combinatorics, to appear, 27 pages.

\bibitem{RG93}
J. Richter-Gebert,
\newblock Oriented Matroids with Few Mutations,
\newblock Discrete Comput Geom 10 (1993), 251-269.

\bibitem{St73} 
R.P. Stanley, 
\newblock Acyclic orientations of graphs,
\newblock Discrete Math.  5  (1973), 171-178.

\bibitem{Tu54}
W.T. Tutte,
\newblock A contribution to the theory of dichromatic polynomials.
\newblock Canadian J. Math., 6 (1954), 80--91.

\bibitem{Wi66}
R.O. Winder,
\newblock Partitions of N-space by hyperplanes.
\newblock SIAM J. Applied Math. 14 (1966), 811-818.

\bibitem{Za75} 
T. Zaslavsky, 
\newblock Facing up to arrangements: Face-count formulas 
for partitions of space by hyperplanes.
\newblock Mem. Amer. Math. Soc. 1 (1975), issue 1, no. 154.



\end{thebibliography}
\end{document}